\documentclass[12pt, reqno]{amsart}
\usepackage{amsmath, amsthm, amscd, amsfonts, amssymb, graphicx, color}
\usepackage[bookmarksnumbered, colorlinks, plainpages]{hyperref}
\hypersetup{colorlinks=true,linkcolor=red, anchorcolor=green, citecolor=cyan, urlcolor=red, filecolor=magenta, pdftoolbar=true}

\newtheorem{theorem}{Theorem}[section]
\newtheorem{lemma}[theorem]{Lemma}
\newtheorem{proposition}[theorem]{Proposition}
\newtheorem{corollary}[theorem]{Corollary}
\theoremstyle{definition}
\newtheorem{definition}[theorem]{Definition}

\theoremstyle{remark}
\newtheorem{remark}[theorem]{Remark}
\numberwithin{equation}{section}

\begin{document}

\setcounter{page}{1}
\title[Convolution dominated operators]{Convolution dominated operators on compact extensions of abelian groups}
\author[G. Fendler , M. Leinert]{Gero Fendler,$^1$$^{*}$ \MakeLowercase{and}
Michael Leinert,$^2$}
\address{$^1$Finstertal 16, D-69514 Laudenbach, Germany.}
\email{\textcolor[rgb]{0.00,0.00,0.84}{gero.fendler@univie.ac.at}}
\address{$^2$Institut f\"ur Angewandte Mathematik, Universit\"at Heidelberg,
Im Neuenhei\-mer Feld 205, D-69120 Heidelberg, Germany.}
\email{\textcolor[rgb]{0.00,0.00,0.84}{leinert@math.uni-heidelberg.de}}

%\dedicatory{This paper is dedicated to Professor ABCD}

\let\thefootnote\relax\footnote{Copyright 2016 by the Tusi Mathematical Research Group.}

\date{Received: xxxxxx; Accepted: zzzzzz.
\newline \indent $^{*}$ Corresponding author}
\newcommand{\norm}[2][]{\| \, {#2} \,\|_{#1}}
\newcommand{\abs}[1]{|{#1}|}
\newcommand{\NN}{\ensuremath{\mathbb{N}}}
\newcommand{\CC}{\ensuremath{\mathbb{C}}}
\newcommand{\RR}{\ensuremath{\mathbb{R}}}
\newcommand{\TT}{\ensuremath{\mathbb{T}}}
\newcommand{\ZZ}{\ensuremath{\mathbb{Z}}}
\newcommand{\spr}[2][]{r_{#1}(#2)}
\newcommand{\projtensor}{\hat{\otimes}}
\newcommand{\supp}[1]{\mbox{supp}(#1)}
\newcommand{\id}{\ensuremath{1}}
\newcommand{\idg}{\ensuremath{e}}
\newcommand{\charfkt}[1]{\chi_{#1}}
\newcommand{\esssup}[2][]{\ensuremath{\mbox{ess~sup}_{#1}{#2}}}
\newcommand{\sign}{\ensuremath{\mbox{sign}}}
\begin{abstract}
If $G$ is a locally compact group, $CD(G)$ the algebra of convolution dominated operators on $L^2(G)$, then an important question is:
Is $\CC\id+CD(G)$ (or $CD(G)$ if $G$ is discrete) inverse-closed in the algebra of
bounded operators on $L^2(G)$?
\par
In this note we answer this question in the affirmative,
provided $G$ is such that one of the following properties is satisfied.
\begin{itemize}
\item[(1)]{There is a discrete, rigidly symmetric, and amenable subgroup 
$H\subset G$ and a 
(measurable) relatively compact 
neighbourhood of the identity $U$, 
invariant under conjugation by elements of $H$, such that 
$\{hU\;:\;h\in H\}$ is a partition of $G$.}
\item[(2)]{The commutator subgroup of $G$ is relatively compact.
(If $G$ is connected, this just means that $G$ is an IN group.)}
\end{itemize}
All known examples where $CD(G)$ is inverse-closed in $B(L^2(G))$ are covered
by this.
\end{abstract}
\maketitle
\section{Introduction}
For an operator on Hilbert space with an additional property,
often this property is not preserved under inversion. So there is an interest
in situations where this does not happen.
E.g. consider on $l^2(\ZZ)$ an operator as a twosided infinite matrix,
then it might have a certain  off-diagonal decay,
i.e its entries $a_{i,j}$ decay as $k=\abs{i-j}$ becomes large. 
A condition of summability like 
$\sum_k\sup\{\abs{a_{i,j}}\,:\,\abs{i-j}=k\}<\infty$ is an example.
This type of condition is preserved under inversion.
Note these operators $A$ are characterised by the condition
that there exists an $\alpha\in l^1(\ZZ)$ 
dominating the operator in the sense that
$\abs{A(\xi)(l)}\leq\sum_k\alpha(k)\abs{\xi(l-k)}$, for example
$\alpha(k)=\sup\{\abs{a_{i,j}}\,:\,\abs{i-j}=\abs{k}\}$.
With canonical operations 
the set of these operators is a Banach $\ast$-algebra.
To see that the set is closed under multiplication
one uses a Fubini type interchange of summation,
which is allowed since we have summable dominants.
An example in Gabor frame theory, where it becomes useful
to consider this class of operators on a non-abelian group,
namely a Heisenberg group with compact centre, is given in \cite{GL01}.
An example relating to mobile communication can be found in \cite{FarrStro10}.
In this note we continue the search for more general groups, where
classes of those operators are preserved under inversion.
\par
Let $G$ be a locally compact group. A bounded operator
$T$ on $L^2(G)$ is called {\em convolution dominated},
if it is dominated
by left convolution with some $L^1$-function, i.e.\ there is $f\in L^1(G)$
such that
$\abs{Tg}(x)\leq f\ast \abs{g}(x)\, \mbox{a.e.}\,\forall g\in L^2(G)$.
\par
The set $CD(G)$ of all convolution dominated operators on
$L^2(G)$ is a $\ast$-sub\-al\-ge\-bra 
of the $\ast$-algebra of all bounded operators $B(L^2(G))$.
In such a situation, an algebra $\mathcal{B}$ and a subalgebra
$\mathcal{A}\subset \mathcal{B}$ with common unit,
the question of inverse-closedness of $\mathcal{A}$ in $\mathcal{B}$
is of importance i.\ e.\ whether an element of $\mathcal{A}$ 
which is invertible in $\mathcal{B}$ must be invertible in $\mathcal{A}$, too.
Probably the first result on inverse-closedness 
is due to N.~Wiener~\cite{Wiener32}
and widely known as Wiener's Lemma:
\par
If a function on the unit circle  with absolutely summable Fourier series
has an inverse with respect to pointwise multiplication in the 
Banach algebra of continuous functions, 
then this inverse has an absolutely summable Fourier series, too.
\par
Using results of Bochner and Phillips~\cite{BP} on operator valued 
Fourier series, 
quite a few authors studied the inverse-closedness
of $CD(G)$  
in $B(L^2(G))$, for abelian discrete groups $G$ 
\cite{GoKaWo89,Bask92,Bask97,Bask97a,Sun05,fgl08}.
Using techniques from non-commutative harmonic analysis 
\cite{hul72,Lep65,Lep67,LP79}, in \cite{fgl08} we together with K.~Gr\"ochenig
treated the  case of rigidly symmetric, amenable, not necessarily abelian, 
discrete groups (which in particular includes all nilpotent discrete groups).
\par
In the case of non-discrete $G$ (here the question is about $\id+T$ in place of $T$,
since $CD(G)$ has no identity) a measurability problem arises~\cite{FGL10},
see also \cite{BelBel15a}. A path avoiding this is to restrict the question
to the algebra $CD_{reg}(G)$ of convolution dominated operators with more
regular side diagonals~\cite{FGL10,BelBel15a}.
In this note, in order to avoid
this restriction, we adopt a different approach 
combining methods of \cite{Kurb90,Kurb99,Kurb01}
with non-commutative harmonic analysis.
With similar methods Farrell and Strohmer~\cite{FarrStro10} 
looked at the generalised Heisenberg groups.
\par
We extend the positive results to the following two classes of groups.
\begin{enumerate}
\item[(1)]{There is a (measurable) relatively compact neighbourhood $U$ of the identity 
and a rigidly symmetric and amenable discrete subgroup 
$H\subset G$ with $hUh^{-1}=U$ for all $h\in H$ such that
$\{hU\}_{h\in H} $ is a partition of $G$.}
\item[(2)]{The topological commutator subgroup of $G$ is compact,
i.e.\ $G$ is a compact extension of an abelian group.
If $G$ is connected, this is equivalent to saying  
that $G$ is an IN group~\cite{Iwasawa51}.}
\end{enumerate}
Note that (1) covers nilpotent Lie groups that admit a rational structure.
The real ``ax+b'' group (for this group the convolution dominated operators are not inverse closed in $B(L^2(G))$~\cite{fl16}) shows that the 
compactness condition in (2) is needed.
Conditions (1) and (2) cover all known examples 
where $CD(G)$ is inverse-closed in $B(L^2(G))$.
We note that groups satisfying property (2) are 
amenable~\cite[Theorem 1.2.6]{green69}.
This is not so obvious in case (1).
In an appendix we show amenability of such groups
by establishing F\o lner's condition.  
\vfill
\section{Preliminaries}
Let $G$ be a locally compact group, $\mathcal{K}(G)$
the space of complex valued functions on $G$ with compact support, and
$dx$ a left Haar measure on $G$.
For a complex-valued function $f$ we denote by
$\overline{f}$ its complex conjugate. 
For a subset $V\subset G$ we denote its closure by $\overline{V}$  and
its Haar measure (provided $V$ is measurable) by $\abs{V}$.
\par
Let $U$ be a (measurable) relatively compact
neighbourhood of the identity $\idg$. 
The following Lemma is well known.
\begin{lemma}\label{lem:overlap}
If $H\subset G$ satisfies $xU\cap yU = \emptyset$ for all $x\not= y$ in $H$,
then for $z\in G$ and relatively compact $K,L \subset G$ the number of all
$h\in H$ with $hL\cap zK\not= \emptyset$ is dominated by
$ \frac{\abs{\overline{KL^{-1}U}}}{\abs{U}}$.
\end{lemma}
\begin{proof}
If $hL$ meets $zK$, we have $h\in zKL^{-1}$, 
hence $hU\subset zKL^{-1}U$. So the number of such elements
cannot exceed 
$\frac{\abs{\overline{zKL^{-1}U}}}{\abs{hU}}=\frac{\abs{\overline{KL^{-1}U}}}{\abs{U}}$.
\end{proof}
 Let $H\subset G$ be a discrete subset
and $U$ a relatively compact neighbourhood of the identity $\idg$ 
such that $\left\{ xU \right\}_{x\in H}$ is a partition of $G$.
With this setting, we define the amalgam space
\[(L^{\infty}, l^1)=\{f\in L^1(G)\; :\; \sum_{k\in H}\norm[\infty]{f\cdot
\charfkt{kU}} < \infty\}.\]
Note that if $U$ is invariant under conjugation by elements of $H$
(this will be our standard assumption below),
then $\abs{xU}=\abs{Ux}, \mbox{ so } \Delta(x)=1 \mbox{ for } x\in H$,
and for $x\in H, u\in U$ we have 
$\Delta(xu)=\Delta(u) \leq \sup_{u\in U}\Delta(u)<\infty$, so $G$ is unimodular.
\begin{proposition}
Given the above assumptions on $U$ and $H$, including the invariance of $U$ under conjugation by elements of $H$, the amalgam space
$(L^{\infty}, l^1)$ is a dense twosided ideal in $L^1(G)$.
\end{proposition}
\begin{proof}
\begin{enumerate}
\item[(a)]{ $\mathcal{K}(G)$ is dense in $L^1(G)$, and
$\mathcal{K}(G)\subset (L^{\infty}, l^1)$, since for $f\in \mathcal{K}(G)$
the number of $h\in H$ with $hU\cap \supp{f}\not= \emptyset$ is at most
$ \frac{\abs{\overline{\supp{f} U^{-1}U}}}{\abs{U}}$.}
\item[(b)]{Since in $L^1$ we have $g\ast f = \sum_{x,y\in H}g\charfkt{xU}\ast f\charfkt{yU}$,
to show that $(L^{\infty}, l^1)$ is a left ideal in $L^1(G)$
it suffices to show that for $f\in (L^{\infty}, l^1), g \in  L^1(G)$, 
$x,y \in H$,
one has $g\charfkt{xU}\ast f\charfkt{yU}\in (L^{\infty}, l^1)$ with
$\norm[(L^{\infty}, l^1)]{g\charfkt{xU}\ast f\charfkt{yU}}\leq \mbox{const}\norm[1]{g\charfkt{xU}}\norm[\infty]{f\charfkt{yU}}$.
\par
Now, $\norm[\infty]{g\charfkt{xU}\ast f\charfkt{yU}}\leq\norm[1]{g\charfkt{xU}}\norm[\infty]{f\charfkt{yU}}$
and $\supp{g\charfkt{xU}\ast f\charfkt{yU}}\subset x\overline{U}y\overline{U}
= xy \overline{U}^2$. The number of all $h\in H$ with 
$hU\cap xy\overline{U}^2\not= \emptyset$
is at most $\frac{\abs{\overline{U}^2\overline{U^{-1}U}}}{\abs{\overline{U}}}=: c$ by Lemma~\ref{lem:overlap} (taking $K=y\overline{U^2}$ and $L=\overline{U}$ there).
So
\[\norm[(L^{\infty}, l^1)]{g\charfkt{xU}\ast f\charfkt{yU}}\leq c \norm[1]{g\charfkt{xU}}\norm[\infty]{f\charfkt{yU}}.\]
\item[(c)]{By the assumptions on $U$ and $H$ the group $G$ is unimodular, so
an argument like the above shows that $(L^{\infty}, l^1)$ is 
a right ideal in $L^1(G)$, too.}}
\end{enumerate}
\end{proof}
\begin{definition}\label{def:cd}
An operator $T\in B(L^2(G))$ is called convolution dominated,
if there is $f\in L^1(G)$ dominating $T$ in the sense that 
\begin{equation}
\abs{Tg}(x)\leq f\ast \abs{g}(x)\, \mbox{a.e.}\,\forall g\in L^2(G).
\end{equation}
Such an $f$ is automatically non-negative.
\par
We denote the algebra of convolution dominated operators by $CD(G)$,
or simply $CD$.
It is normed by
\begin{equation}
\norm[CD]{T}=\inf \{\norm[L^1(G)]{f}\,:\,\abs{Tg}(x)\leq f\ast \abs{g}(x)\, \mbox{a.e.}\,\forall g\in L^2(G)\},
\end{equation}
where $T\in CD$.
We  denote by  $CD_{\infty}$ the space of all 
convolution dominated operators $T$ 
on $L^2(G)$ which are dominated by convolution with some 
$f\in (L^{\infty}, l^1)$. The norm of $T\in CD_{\infty}$ is defined by
\begin{equation}\label{eq:CD-infty}
\norm[CD_{\infty}]{T}=\inf \{\norm[(L^{\infty},l^1)]{f}\,:\,\abs{Tg}(x)\leq f\ast \abs{g}(x)\, \mbox{a.e.}\,\forall g\in L^2(G)\}.
\end{equation}
\end{definition}
From Proposition 2.3 of \cite{fl16} 
we know that any convolution dominated 
operator is an integral operator with respect to a kernel.
Calling kernels equivalent if they coincide l.a.e on $G\times G$,
we have a linear bijection between the convolution dominated 
operators and the equivalence classes of kernels satisfying (\eqref{eq:cd}). 
A  kernel $t$ of such an operator $T$ satisfies 
\begin{equation}\label{eq:kernel}
T(g)(x) = \int_G t(x,y) g(y)\,dy,\;\mbox{l.a.e.},\;\forall g\in L^2(G)
\end{equation} 
and
\begin{equation}\label{eq:cd}
\abs{t(x,y)}\leq f(xy^{-1})\quad\mbox{l.a.e.}
\mbox{ for some }f\in L^1(G).
\end{equation} 
At the level of kernels the composition of convolution dominated
operators $S,T$ with respective kernels $s,t$ is given by convolution of kernels
\begin{equation}
s\ast t (x,y)= \int_G s(x,z)t(z,y)\,dz\quad\mbox{l.a.e.},
\end{equation} 
This formula makes sense because $S$ and $T$ are dominated by convolution with
integrable functions.
\par
In case that the operator $T$ is in  $CD_{\infty}$,
we can take the dominating $f$ in $(L^{\infty},l^1)$.
The argument in Remark 2.4 of \cite{fl16} 
shows that the infimum 
in (\eqref{eq:CD-infty}) is attained. Actually the function
$\sum_{i\in H}\esssup[{xy^{-1}\in iU}]{\abs{t(x,y)}}\cdot \charfkt{iU}$
does the job.
\begin{remark}
If $t$ is a kernel for $T$, $\abs{t(x,y)}\leq f(xy^{-1})\;\mbox{l.a.e.}$
for some $f\in L^1(G)$, then 
$N:=\{\,(x,y)\,|\,\abs{t(x,y)}>f(xy^{-1})\,\}$
is a local null set, so 
$t':=t\charfkt{U\times U \setminus N}$ 
is equivalent to $t$ and hence defines the same operator, and 
$t'(x,y)\leq f(xy^{-1})$ everywhere. So, replacing $t$ by $t'$,
we may replace ''l.a.e.'' by ''a.e.'' in \eqref{eq:kernel} and \eqref{eq:cd}.
\end{remark}
\begin{proposition}
$CD_{\infty}$ is a dense ideal in $CD$.
\end{proposition}
\begin{proof}
\begin{enumerate}
\item[(a)]{Let $T\in CD$
with convolution kernel $k$, and $f\in L^1(G)$ be given, 
where $\abs{k(x,y)}\leq f(xy^{-1})\quad\mbox{l.a.e.}$ . 
For $\varepsilon>0$ there is some $0\leq g\in \mathcal{K}(G)$ with
$\norm[1]{f-g}< \varepsilon$. 
Let $k_{\infty} := \sign\ k \cdot (\abs{k}\wedge M(g))$\footnote{
For $z\in \CC$ let $\sign\ z = \frac{z}{\abs{z}} \mbox{ if }z\not= 0,
\mbox{ resp. } 0 \mbox{ if } z=0$, and extend this pointwise to complex valued functions.}, 
where $M(g)(x,y):=g(xy^{-1})$.
Then $\abs{k-k_{\infty}}\leq M(\abs{f-g})$. So, if $T_{\infty}$ 
is defined by $k_{\infty}$, we have $T_{\infty}\in CD_{\infty}$
and $\norm[CD]{T-T_{\infty}}\leq \norm[1]{f-g}<\varepsilon$.}
\item[(b)]{For $T\in CD, \; S\in CD_{\infty}$ with convolution kernels
$t$ and $s$, respectively, there are $f\in L^1(G), g\in (L^{\infty}, l^1)$
with $\abs{t}\leq M(f), \abs{s}\leq M(g) \mbox{ l.a.e. }$.
So 
\[\abs{t\ast s(x,y)}=\abs{\int_G t(x,z)s(z,y)\,dz}\leq\int_G f(xz^{-1})%
g(zy^{-1})\,dz=f\ast g (xy^{-1}).\]
So $TS$ is dominated by convolution with $f\ast g \in (L^{\infty}, l^1)$,
hence $TS \in CD_{\infty}$.}
\item[(c)]{Analogously we see that $CD_{\infty}$ is a right ideal in $CD$.} 
\end{enumerate}
\end{proof}
\begin{proposition}
If $G$ is nondiscrete, then $CD_{\infty}$ has no identity.
\end{proposition}
\begin{proof}
Suppose $G$ is nondiscrete and $E$
is the identity of $CD_{\infty}$, with corresponding kernel $e$.
If $\mathcal{V}$ is the downward directed system of compact neighbourhoods
of the group identity, 
and $e_V:=\frac{1}{\abs{V}}\chi_V$ for $V\in \mathcal{V}$,
then $\{e_V\}_{V\in \mathcal{V}}$
is an approximate identity of $L^1(G)$.
Since $E\in B(L^2(G))$, we have 
$\norm[2]{E(e_V\ast g)-Eg}\to 0$ for $g\in \mathcal{K}(G)$.
Denoting $E_V$ the operator belonging to the kernel
$(x,y)\mapsto e_V(xy^{-1})$, using Fubini one has 
$E(e_V\ast g)=(EE_V)g=E_Vg$, hence 
\[\int_G\int_G e_V(xy^{-1})g(y)f(x)\,dydx\to \int_G\int_G e(x,y) g(y)f(x)\,dydx
\mbox{ for } f\in\mathcal{K}(G).\]
If $f\otimes g$ is such that its support does not meet the diagonal
$\{(x,x^{-1})\,|\, x\in G\}$, then the left hand side vanishes for 
sufficiently small $V$, so the right hand side is $0$.
This implies that $e=0$ locally a.e.\ outside the diagonal.
$G$ being nondiscrete, the diagonal is a local null set, so $e=0$ l.a.e.,
hence $E=0$ which contradicts $E(CD_{\infty})=CD_{\infty}\not= \{0\}$.  
\end{proof}
\begin{remark}\label{rem:identity}
So $CD_{\infty}(G)$ has an identity if and only if $G$ is discrete.
\end{remark}
Since $CD_{\infty}(G)$ is a dense ideal in $CD(G)$, Remark~\ref{rem:identity}
holds true for $CD(G)$ too.
\section{Matrices of operators and kernels}\label{sec:matrices}
Now we shall decompose $A\in CD_{\infty}$ and its convolution kernel 
$ a=\left( a(x,y) \right)_{x,y\in G}$
as a matrix of blocks.
Since $ L^2(G)={{\oplus_{i\in H}}} L^2(iU)$ 
(orthogonal sum of Hilbert spaces), we may divide $A$ into blocks
$A_{ij}$, where $A_{ij}\in B(L^2(jU),L^2(iU))$ is the 
restriction of $A$ 
to $L^2(jU)$ composed with the orthogonal projection onto $L^2(iU)$.
If we order the finite subsets of $H\times H$ by inclusion,
we have $ A= \left(A_{ij}\right)_{i,j \in H}$ in the sense
that the finite submatrices of $\left(A_{ij}\right)_{i,j \in H}$,
when interpreted as operators on $L^2(G)$, converge to $A$ 
in the strong operator topology.
If $ B= \left(B_{ij}\right)_{i,j \in H}$,
then $ AB = \left((AB)_{i,j}\right)$
where $\left(AB\right)_{i,j}=\sum_{k\in H}A_{ik}B_{kj}$,
since multiplication on bounded sets of operators
is strongly continuous.
So the map $ A\mapsto \left(A_{ij}\right)$ is an algebra isomorphism from $CD_{\infty}$ onto its image.
The map for the corresponding kernels reads $a\mapsto (a_{ij})$ where $a_{ij}$
is the restriction of $a$
to $iU\times jU$,  i.e.\ $a_{ij}=(a(x,y))_{(x,y)\in iU\times jU}$.
Note that $\norm[Op]{A_{ij}}\leq \abs{U}\norm[\infty]{a_{ij}}$,
where the infinity norm is taken on $iU\times jU$ 
with respect to product Haar measure.
Since for $A,B\in CD_{\infty}$ the respective kernels $a,b$ are dominated by 
$L^1$-functions, the convolution $a\ast b$ is the kernel 
corresponding to $AB$. (This is done with a Fubini argument, which is not valid
for general kernels.)
Denoting $\lambda$ the left regular representation of $G$ on $L^2(G)$
and $\lambda_i=\lambda(i),\quad i\in H$,
we define a Hilbert space isomorphism
$ S:{{\oplus_{i\in H}}} L^2(U)\to {{\oplus_{i\in H}}} L^2(iU)$
by $S((u_i)_{i\in H})=(\lambda_i u_i)_{i\in H}$.
Then $A^{\circ}_{ij}:= \lambda^{-1}_i A_{ij}\lambda_j\in B(L^2(U))$ and 
$\norm[op]{A^{\circ}_{ij}}=\norm[op]{A_{ij}}$.
We have $A^{\circ}_{ik} A^{\circ}_{kj}=\lambda^{-1}_i A_{ik}\lambda_k%
\lambda^{-1}_kA_{kj}\lambda_j=\lambda^{-1}_i A_{ik}A_{kj}\lambda_j=%
(A_{ik} A_{kj})^{\circ}$. For the kernels this reads 
$(a_{ik}\ast a_{kj})^{\circ}=a^{\circ}_{ik}\ast a^{\circ}_{kj}$, 
and $a^{\circ}_{ij}=(a(i\xi,j\eta))_{\xi,\eta\in U}$.
Thus multiplication of blocks is carried into operator composition in
$B(L^2(U))$, respectively convolution of kernels on $U\times U$.
Altogether we obtain that the map 
$ A\mapsto \left(a^{\circ}_{ij}\right)_{i,j\in H}$ 
is an algebra isomorphism onto its image in the kernel-valued
matrices with matrix multiplication, where the multiplication of entries is
convolution of kernels on $U\times U$.
If we define the involution 
$\left(a^{\circ}_{ij}\right)_{i,j\in H}^{\ast}=%
\left(a^{\circ {\ast}}_{ji}\right)_{i,j\in H}$,
where $b^{\ast}(x,y)=\overline{b(y,x)}$ for any kernel $b$ on 
$U\times U$ and $x,y\in U$, then the map $ A\mapsto \left(a^{\circ}_{ij}\right)_{i,j\in H}$ preserves the involution, too.
\begin{remark}\label{rem:decomposition}
The reader will have noticed that,
if we allow ourselves to identify the isomorphic Hilbert spaces $L^2(G)$
and $\oplus_{i\in H}L^2(U)$ and interpret the matrix
$\left(a^{\circ}_{ij}\right)$ as an operator (in the canonical way),
then $A\mapsto \left(a^{\circ}_{ij}\right)$ is the identity map, i.e.\ the
operator defined by $\left(a^{\circ}_{ij}\right)$ is the original $A$ again.
\end{remark}
For $G, H$, and $U$ as above we define different kinds of diagonals on $G\times G$.
\begin{definition}\label{def:abd}
For $k\in H$ we call 
$\{ (x,y)\in G\times G\, |\, xy^{-1} \in kU\}$ the {\em band diagonal}
determined by $k$, the set $\cup_{ij^{-1}\in kU}(iU\times jU)$
the {\em approximate block diagonal} determined by $k$, and 
$\cup_{ij^{-1} = k}(iU\times jU)$ the {\em block diagonal} determined by $k$.
\end{definition}
\begin{lemma}\label{lem:blockdiagonal}
There is $n \in \NN$ such that each approximate block diagonal
meets at most $n$ band
diagonals,
and conversely each band diagonal meets at most $n$ approximate block diagonals
(and hence can be covered by these).
\end{lemma}
\begin{proof}
Let $x,y\in G$. There are $s,t\in H$ and $\xi,\eta\in U$ 
with $x=s\xi$ and $y=t\eta$.
If $(x,y)$ is in the band diagonal determined by $k\in H$, 
this means $xy^{-1}\in kU$ i.e.\ $s\xi\eta^{-1}t^{-1}\in kU$ or 
$st^{-1}\in k U^2U^{-1}$.
If $(x,y)$ is in the approximate block diagonal determined by $l\in H$, 
this means $st^{-1}\in lU$.
So if this approximate block diagonal meets the above band diagonal, this means 
$l\in kU^2U^{-2}$ or $lU^2 \cap kU^2\not= \emptyset$.
For fixed $l$ (resp. $k$), the number of possible such $k$ (resp. $l$)
is dominated by 
$\frac{\abs{\overline{U^2U^{-2}U}}}{\abs{U}}$ by Lemma~\ref{lem:overlap}.
\end{proof}
\begin{corollary}\label{cor:vm}
$CD_{\infty}$ 
is bicontinuously $\ast$-isomorphic
to the algebra of kernel  valued matrices
$CD_H:=\{ (a^{\circ}_{ik})_{i,k\in H}\}$
with norm\\ 
$\norm{(a^{\circ}_{ik})_{i,k \in H}}:=\sum_{l\in H}\sup_{ik^{-1}\in lU}\norm[\infty]{a^{\circ}_{ik}}$.
\end{corollary}
\begin{proof}
As seen before Remark~\ref{rem:decomposition}, the map
$A\mapsto \left(a^{\circ}_{ij}\right)_{i,j\in H}$ carries the algebra 
$CD_{\infty}$ isomorphically onto its image in the kernel valued matrices. 
Lemma~\ref{lem:blockdiagonal} shows that this image is precisely $CD_H$
and that there are norm estimates both ways for this isomorphism.
\end{proof}
\begin{remark}
If $H$ is a subgroup, approximate block diagonals are block diagonals.
\end{remark}
\begin{proof}
If $i,j,k\in H$ with $ij^{-1}\in kU$, then $k$ is the only element of $H$ in 
$kU$, since $lU\cap l'U = \emptyset$ for  $l \not= l'$ in $H$.
It follows that $ij^{-1}=k$.
\end{proof}
\section{Spectrality of $CD_{\infty}$ and $CD_H$.}
\label{sec:subgroup}
Now we assume that $H$ is a (discrete) rigidly symmetric and amenable
subgroup of $G$ 
and that the Haar measure of $G$ is normalised
such that $\abs{U}=1$. Let $\mathcal{A}=l^{\infty}(H,L^{\infty}(U\times U))$
denote the space of all bounded functions $f: H\to L^{\infty}(U\times U)$
with pointwise linear operations, multiplication
$(f,g)\mapsto fg$, where $fg(h)=f(h)\ast g(h)$ 
(where $\ast$ denotes the convolution of kernels)
and involution $f\mapsto f^{\ast}$, where 
$f^{\ast}(h)(u,v)=\overline{f(h)(v,u)}$, endowed with the norm
$\norm[\mathcal{A}]{f}=\sup_{h\in H}\norm[\infty]{f(h)}$.
Then $\mathcal{A}$ is a Banach $\ast$-algebra.
\par
We denote the left regular representation
of $H$ on  $\mathcal{A}$ by $T$. So $(T_kf)(h) = f(k^{-1}h)$
for $f\in\mathcal{A}$ and $h,k \in H$.
The twisted $L^1$ algebra $\mathcal{L}=l^1(H,\mathcal{A},T)$ in the sense of Leptin
is the Banach space of all functions
$F:H\to \mathcal{A}$ with product
$F\star G(h)=\sum_{y\in H} T_yF(hy)G(y^{-1})$,
involution $F\mapsto F^{\ast}$, where $F^{\ast}(h)=T^{-1}_hF(h^{-1})^{\ast}$,
and norm $\norm{F}=\sum_{h\in H}\norm[\mathcal{A}]{F(h)}$.
\begin{theorem}\label{thm:CD}
The Banach $\ast$-algebra $CD_H$ is isometrically $\ast$-isomorphic to
$\mathcal{L}=l^1(H,l^{\infty}(H,L^{\infty}(U\times U)),T)$.
As a result $CD_{\infty}$ is bicontinuously $\ast$-isomorphic to $\mathcal{L}$.
\end{theorem}
\begin{proof}
Like in \cite{fgl08} we define a representation $R$ of $\mathcal{L}$, 
but this time
on $L^2(G)=\bigoplus_{i\in H} L^2(U)$.
The image of $R$ will turn out to be $CD_H$.
If $\delta^m_h$ denotes the $\mathcal{A}$-valued Dirac function which takes
the value $m\in l^{\infty}(H,L^{\infty}(U\times U))$ at $h$ and vanishes on
$H\setminus\{h\}$, we set $R\delta^m_h=\lambda_h\circ M_m$, where
$M_m$ is the multiplication operator
$(\xi_i)_{i\in H}\mapsto (m(i)\ast \xi_i)_{i\in H}$,
where of course $(\xi_i)_{i\in H}\in \bigoplus_{i\in H} L^2(U)=L^2(G)$.
Then $R\delta^m_h$ coincides with the operator $T$
given by the matrix $(t_{ij})\in CD_H$ with zero entries
outside the $h$ diagonal $ij^{-1}=h$ and $t_{hj,j}=m(j)$ for $j\in H$.
To see this, it suffices to apply both operators to $(\delta_{i,k}\xi)_{i\in H}$, 
where $\xi\in L^2(U)$, $k\in H$.
We have $T(\delta_{i,k}\xi)=(\eta_i)$, where $\eta_i=\delta_{i,hk} m(k)\ast \xi$,
which equals $R\delta^m_h(\delta_{i,k}\xi)$
(note that $\lambda_h$ permutes the $L^2(U)$-blocks of $L^2(G)$,
shifting the $k$ block to the $hk$ block).
Clearly $\norm[\mathcal{L}]{\delta^m_h}=\norm[CD_H]{T}=\norm[\infty]{m}$.
Extending $R$ by linearity and continuity we obtain an isometric
linear isomorphism from $\mathcal{L}$ onto $CD_H$.
Consider $\delta^m_h$ and $\delta^n_k$ with $h,k\in H$,
$m,n\in \mathcal{A}= l^{\infty}(H,L^{\infty}(U\times U))$.
\begin{eqnarray*}
R\delta^m_h\;R\delta^n_k&=&\lambda_hM_m\lambda_kM_n\\
&=&\lambda_{hk}\lambda_k^{-1}M_m\lambda_kM_n\;=\;%
\lambda_{hk}M_{(T^{-1}_{k}m)n}\\
&=&R\delta^{(T^{-1}_{k}m)n}_{hk}\;=\;R(\delta^m_h\ast\delta^n_k).
\end{eqnarray*}
This implies that $R$ is multiplicative.
Finally 
\begin{eqnarray*}
(R\delta^m_h)^{\ast}&=& (\lambda_hM_m)^{\ast}\;=\;M_{m^{\ast}}\lambda_h^{-1}\\
&=&\lambda_h^{-1}M_{T_hm^{\ast}}\;=\;R(\delta^{T_hm^{\ast}}_{h^{-1}});=\;R((\delta^m_h)^\ast).
\end{eqnarray*}
So $R$ is a $\ast$-isomorphism from 
$\mathcal{L}=l^1(H,l^{\infty}(H,L^{\infty}(U\times U)))$
onto $CD_H$, respectively $CD_{\infty}$, which is isometric, respectively bicontinuous. At the same time $R$ is a $\ast$-representation of $\mathcal{L}$
on $L^2(G)$.
\end{proof}
\begin{lemma}\label{lem:K2}
Let $K_{2,\infty}$ be the space of (equivalence classes of) measurable kernels
on $U\times U$ with 
\[ \norm[2,\infty]{k}=\esssup[y\in U]{\norm[2]{k(\cdot,y)}<\infty}.\]
Then $L^{\infty}(U\times U)$ is a right Banach $K_{2,\infty}$ module, and
 $K_{2,\infty}$ is a left $L^2(U\times U)$ module for the convolution of kernels.
\end{lemma}
\begin{proof}
In order to avoid additional constants coming up, we normalise
Haar measure on $G$ so that $\abs{U}=1$.
\par
\begin{enumerate}
\item[(a)]{Let $g\in L^{\infty}(U\times U)$ and $k\in K_{2,\infty}$. Since
\begin{eqnarray*}
\abs{g\ast k\,(r,s)}&=&\abs{\int_U g(r,t)k(t,s)\,ds}\;\leq
\;\norm[\infty]{g(r.\cdot )} \norm[1]{k(\cdot,s)}\\
&\leq& \norm[\infty]{g}\norm[2]{k(\cdot,s)}\;\leq\;
\norm[\infty]{g}\norm[2,\infty]{k}\quad \mbox{ a.e.\ },
\end{eqnarray*}
we have
\begin{equation*}
\norm[\infty]{g\ast k} \leq \norm[\infty]{g}\norm[2,\infty]{k}.
\end{equation*}}
\item[(b)]{Let $h\in L^2(U\times U)$ and $k\in K_{2,\infty}$. Then
\begin{eqnarray*}
\abs{h\ast k(r,s)}&\leq& \norm[2]{h(r,\cdot)}\norm[2]{k(\cdot,s)}\;\leq\;
\norm[2]{h(r,\cdot)}\norm[\infty,2]{k}\quad\mbox{a.e.\ },
\end{eqnarray*}
so
\begin{eqnarray*}
\norm[2]{h\ast k(\cdot,s)}^2&\leq& \int_U\int_U \abs{h(r,u)}^2\,du\,dr
\norm[2,\infty]{k}^2\;\leq\;\norm[2]{h}^2\norm[2,\infty]{k}^2
\end{eqnarray*}
Hence
\begin{equation*}
\norm[2,\infty]{h\ast k}\leq\norm[2]{h}\norm[2,\infty]{k}
\end{equation*}
}
\end{enumerate}
\end{proof}
\begin{definition}
A subalgebra $\mathcal{A}$ of an algebra $\mathcal{B}$ is called 
a spectral subalgebra of $\mathcal{B}$ 
or spectral in $\mathcal{B}$, if for every $a \in \mathcal{A}$
 the spectrum of $a$ in $ \mathcal{A}$ coincides 
with its spectrum in $\mathcal{B}$ except perhaps for the value zero. 
(Without the removal of zero the notion would be equivalent to 
inverse-closedness).
\end{definition}
In the following remark for an element $a$ of a Banach algebra $\mathcal{A}$
we denote
its spectral radius by $\spr{a}$.
\begin{remark}
\begin{enumerate}
\item[]
\item[(a)]
$L^{\infty}(U\times U)$ and the Hilbert--Schmidt kernels $L^2(U\times U)$
are Banach algebras for convolution, and so is $K_{2,\infty}$, since 
$\norm[2,\infty]{~}\geq \norm[2]{~}$ on $K_{2,\infty}$.\\
\item[(b)]
Since in general one sided ideals are spectral subalgebras,
$(L^{\infty}(U\times U),\ast)$ is spectral in $K_{2,\infty}$
which is spectral in the algebra of Hilbert--Schmidt operators on $L^2(U)$,
which in turn is spectral  in $B(L^2(U))$.
So $(L^{\infty}(U\times U),\ast)$
is spectral in $B(L^2(U))$.\\
\item[(c)]
The closure of $L^{\infty}(U\times U)$
in $B(L^2(U))$ is its $C^{\ast}$-hull $C^{\ast}(L^{\infty}(U\times U))$:
For $a\in L^{\infty}(U\times U)$ and the operator $A$ defined by it, one has 
$\norm{A}=\spr{A^{\ast}A}^{\frac{1}{2}}=\spr{a^{\ast}a}^{\frac{1}{2}}$ by (b).
Since $\norm{\pi(a)}\leq\spr{a^{\ast}a}^{\frac{1}{2}}$ for any
Hilbert space $\ast$-representation of $L^{\infty}(U\times U)$,
the norm $a\mapsto \norm{A}$ is the greatest $C^{\ast}$-seminorm on it.
Hence the $C^{\ast}$-hull of $L^{\infty}(U\times U)$ is 
its closure in $B(L^2(U))$.\\
\item[(d)]
The argument in (c) 
shows that if $\mathcal{A}$ is a Banach $\ast$-algebra contained in a $C^{\ast}$-algebra $\mathcal{B}$ and $\mathcal{A}$
is spectral in $\mathcal{B}$, then the closure of $\mathcal{A}$ in
$\mathcal{B}$ is the $C^{\ast}$-hull $C^{\ast}(\mathcal{A})$.
(Without the spectrality assumption this is false. Consider
$L^1(G)\subset B(L^2(G))$ for a non-amenable group $G$.)
\end{enumerate}
\end{remark}
The set $\mathcal{C} = l^{\infty}(H,B(L^2(U))$ with 
pointwise operations and involution 
(where multiplication in $B(L^2(U))$ is composition) is a $C^{\ast}$-algebra. 
The left regular representation of $H$ on $\mathcal{C}$ is denoted by $T$. 
So $T_kf(h) = f(k^{-1} h)$ for $f$ in $\mathcal{C}$ and $h,k \in H$. 
In analogy to the beginning of this section we consider the twisted 
$L^1$-algebra $\mathcal{B} = l^1(H,l^{\infty}(H,B(L^2(U)),T)$ and define the 
$\ast$-representation $R$ of $\mathcal{B}$ like 
for $\mathcal{L}$ in the proof of 
Theorem~\ref{thm:CD}.
\begin{proposition} 
\begin{enumerate}
\item[]
\item[(a)]
The closure of $R(\mathcal{B})$ in $B(L^2(G))$ is  its
 $C^{\ast}$-hull $C^{\ast}(\mathcal{B})$.\\
\item[(b)]
$R(\mathcal{B})$ is spectral in $B(L^2(G))$.\\
\item[(c)]
$CD_{\infty}(G)$ is spectral in $B(L^2(G))$.
In particular, $CD_{\infty}(G)$ is a symmetric
Banach $\ast$-algebra.\\
\end{enumerate}
\end{proposition}
\begin{proof}
Like in the proof of Theorem~\ref{thm:CD} we let 
\[M:l^{\infty}(H,B(L^2(U)))\to B(l^2(H,L^2(U)))=B(L^2(G))\]
be the $\ast$-representation $c\mapsto M_c$, where
$M_c(\xi_i)_{i\in H}=(c(i)\xi_i)_{i\in H}, \;\xi\in l^2(H,L^2(U))$.
Since $H$ is amenable and $l^{\infty}(H,B(L^2(U)))$ is
a $C^{\ast}$ algebra, by Leptin \cite{lep68} the
$M$-regular representation is a maximal representation of $\mathcal{B}$.
It is weakly equivalent to $R$.
\par
To see this one slightly modifies the proof of \cite[Prop 3]{fgl08}. 
Since $R$ happens on $L^2(G)=l^2(H,L^2(U))$ and 
$\lambda^M$ happens on $l^2(H,L^2(G))=l^2(H\times H,L^2(U))$,
we work on this last space, which we denote $\mathcal{H}$ for short.
Like in \cite{fgl08} let $R^{\omega}$ denote the extension of $R$
from $L^2(G)$ to $\mathcal{H}$ defined by letting the operator
$R(f)=\sum_{y\in H}\lambda(y)\circ M_{f(y)}$, where $f\in \mathcal{L}$,
act on the first coordinate only.
That is
\begin{equation*}
R^{\omega}(f)\xi(x,z)=\sum_{y\in H}f(y)(y^{-1}x)\xi(y^{-1}x,z) 
\mbox{ for } \xi\in \mathcal{H},\; x,z\in H.
\end{equation*}
The operator $S:\mathcal{H}\to \mathcal{H}$
defined by $S\xi(x,z)=\xi(xz,z)$ is unitary and intertwines
$R^{\omega}$ and $\lambda^M$ as in \cite{fgl08}.
\par
So $R^{\omega}$, and in turn $R$ are  maximal representations of $\mathcal{L}$..
Thus the closure of $R(\mathcal{B})$ in $B(L^2(G))$ is $C^{\ast}(\mathcal{B})$.
This proves (a).
\par
For $f\in\mathcal{B}$ the operator $R(f)$ can be viewed
as an $l^1$-sum of its diagonals.
Each diagonal (with zero entries outside the diagonal)
is a bounded operator on $L^2(G)$, and its operator norm is the supremum
of the
$B(L^2(G))$-norms of its entries. 
So the $\ast$-isomorphism $f\mapsto R(f)$ maps $\mathcal{B}$
isometrically into $l^1(H,B(L^2(G))$. In particular,
$R(\mathcal{B})$ is a complete and hence closed $\ast$-subalgebra of 
$l^1(H,B(L^2(G)))$ $\tilde=l^1(H)\projtensor B(L^2(G))$. The latter
is symmetric, since $H$ is rigidly symmetric. So $R(\mathcal{B})$
is symmetric, too.
By \cite{fgl08} and (a), $R(\mathcal{B})$ is spectral in $B(L^2(G))$.
This proves (b).
\par
By Lemma \ref{lem:K2}, $\mathcal{L}$
is a right ideal in $l^1(H,l^{\infty}(H,K_{2,\infty}))$, which is a left
ideal in $l^1(H,l^{\infty}(H,L^2(U\times U)))$, which is a twosided
ideal in $\mathcal{B}\tilde=R(\mathcal{B})$ which by (b) is spectral
in $B(L^2(G)))$. So $CD_{\infty}(G)\tilde=R(\mathcal{L})\tilde=\mathcal{L}$
is spectral in $B(L^2(G))$.
\end{proof}
Let $\id$ denote the identity operator in $B(L^2(G))$.
\begin{corollary}\label{cor:invcl}
$\CC \id + CD_{\infty}(G)$ is inverse-closed in $B(L^2(G))$.
\end{corollary}
\begin{proof}
If $G$ is discrete,  this is contained in~\cite{fgl08}
($\CC\id$ may be omitted here, since $\id\in CD_{\infty}=CD$).
Hence we may assume that $G$ is non-discrete. Then $CD_{\infty}(G)$
has no identity and is spectral in $B(L^2(G))$ which has the identity $\id$.
So the assertion follows.
\end{proof}
\begin{theorem}\label{thm:invcl}
$\CC \id + CD(G)$ is inverse-closed in $B(L^2(G))$.
\end{theorem}
\begin{proof}
If $G$ is discrete then $CD_{\infty}(G)=CD(G)$, so this case is covered
by Corollary~\ref{cor:invcl}.
In the non-discrete case:
$\CC \id +CD_{\infty}(G)$ is inverse-closed in $B(L^2(G))$, and $CD_{\infty}(G)$ 
is a
dense twosided ideal in  $CD(G)$. (See for instance \cite{Kurb01} for the argument.)
\end{proof}
\section{Groups with compact commutator}
Consider a locally compact group $G$ with compact
topological commutator subgroup $C=\overline{[G,G]}$.
We denote $p_C:G\to G/C$ the canonical projection.
$G/C$ is a locally compact abelian group, so
$G/C=\RR^d \times J$, where $d\in \NN\cup\{0\}$ and $J$ contains 
a compact open subgroup $K$.
Let $R\subset J$ be a complete set of representatives of $J/K$,
$U:= \left( p_C^{-1}([-\frac{1}{2},\frac{1}{2})^d\times K \right)$, 
and let $H\subset G$
be a complete set of representatives of $\ZZ^d\times R\subset G/C$.
Then $U$ is relatively compact, measurable and $H$-invariant.
Furthermore $\{hU\}_{h\in H}$ is a partition of $G$. The set $hU$
does not depend on the choice of the representatives
in $J$ and in $G$, so we may 
write $\dot{h}U$ for $hU$ with $\dot{h}\in D:=\ZZ^d\times J/K$.
\par
Every $h\in H$ is of the form $h=rep_C(m,r)$ where $rep_C$ denotes
a representative of $(m,r)\in G/C$, $m\in \ZZ^d$, and $r$ is a representative
$rep_K(rK)$ of $rK\in J/K$ in $J$.
So $hU=p_C^{-1}((m+[-\frac{1}{2},\frac{1}{2})^d)\times rK)$
which means that $hU$
only depends on $\dot{h}:=(m,rK)=(id\times p_K)(p_C(h))\in D$.
If $k=rep_C(m',r')$ and $\dot{h}=\dot{k}$, it follows that 
$m=m'$ and $rK=r'K$, hence $h=k$. So the map $h\mapsto \dot{h}$
is bijective from $H$ to $D$.
Note also that $(hk\dot{)}=\dot{h}\dot{k}$ and $(h^{-1}\dot{)}=(\dot{h})^{-1}$
(but unlike $D$ the set $H$ need not be a group).
\begin{proposition}\label{prop:abd}
The approximate block diagonals
$({hU,kU})_{hk^{-1}\in lU}$ of Definition~\ref{def:abd} 
are exactly the block diagonals 
$({\dot{h}U,\dot{k}U})_{\dot{h}\dot{k}^{-1}=\dot{l}}$~.
\end{proposition}
\begin{proof}
\begin{enumerate}
\item[(a)]{$hk^{-1}\in lU$ implies $l^{-1}hk^{-1}\in U$.
The formula for $h\mapsto \dot{h}$ is meaningful on all of $G$
and defines an homomorphism onto $\RR^d\times J/K$.
Now $\dot{l}^{-1}\dot{h}\dot{k}^{-1}\in \dot{U}\cap D=\{0\}$
(where we denoted $\{\dot{u} : u\in U\}$ by $\dot{U}$). 
So  $\dot{h}\dot{k}^{-1}=\dot{l}$.}
\item[(b)]{If $\dot{h}\dot{k}^{-1}=\dot{l}$, then $hk^{-1}$ differs from  
$l$ by an element $w\in G$ with $\dot{w}=0$, 
i.e.\ $w\in p_c^{-1}(\{0\}\times K )$.
So $hk^{-1}=lw\in lU$.}
\end{enumerate}
\end{proof}
Corollary~\ref{cor:vm} and Proposition~\ref{prop:abd} now imply:
\begin{theorem}
If the (topological) commutator subgroup $\overline{[G,G]}$ is compact,
then $CD_{\infty}$ is bicontinuously $\ast$-isomorphic to the 
$\ast$-algebra of kernel-valued matrices 
\[CD_H=\left\{ (a^{\circ}_{hk})_{h,k \in H}\right\}=%
\left\{(a^{\circ}_{\dot{h}\dot{k}})_{\dot{h},\dot{k}\in D}\right\},\]
with norm
\[\norm{(a^{\circ}_{\dot{h}\dot{k}})_{\dot{h},\dot{k}\in D}}%
=\sum_{\dot{l}\in D}\sup_{\dot{h}\dot{k}^{-1}=\dot{l}}\norm[\infty]{a^{\circ}_{\dot{h}\dot{k}}},\]
the infinity norm being taken on $hU\times kU$ with product  Haar measure on it.
\end{theorem}
Like in the beginning of section~\ref{sec:matrices}
we have 
$L^2(G)=\oplus_{h\in H}L^2(hU)=\oplus_{h\in H}L^2(U)$
with canonical isomorphisms, and as we just have seen we may replace 
the labelling $h\in H$ by $\dot{h}\in \dot{H}=D$.
Although $D=\ZZ^d\times J/K$ is not a subgroup of $G$, in general, it acts on 
$ L^2(G)=\oplus_{\dot{h}\in D}L^2(U)$ by permuting the $L^2(U)$ blocks.
Denoting
the action of $\dot{l}\in D$ by $\lambda_{\dot{l}}$ we have 
$\lambda_{\dot{l}}(f_{\dot{h}})_{\dot{h}\in D}=(f_{\dot{l}^{-1}\dot{h}})_{\dot{h} \in D}$
for $(f_{\dot{h}})_{\dot{h}\in D}\in \oplus_{\dot{h}\in D}L^2(U)$.
At the level of $L^2(G)=\oplus_{\dot{h}\in H}L^2(hU)$
we would have 
$\lambda_{\dot{l}}g_h=\lambda_{r(\dot{l}\dot{h})}\lambda_{h^{-1}}g_h=\lambda_{r(\dot{l}\dot{h})h^{-1}}g_h$ 
for $g_h\in L^2(hU)$, where $r(\dot{l}\dot{h})$ 
denotes the unique $k\in H$ with $\dot{k}=\dot{l}\dot{h}$.
This is a consequence of respecting our identifications.
Note that $\lambda_{\dot{l}}$ may be different from $\lambda_{l}$.
\par
We also define the action $T$ of $D=\dot{H}$ on
$\mathcal{A}=l^{\infty}(H,L^{\infty}(U\times U))=$\linebreak[4] $l^{\infty}(D,L^{\infty}(U\times U))$
by $T_{\dot{k}}f(\dot{h})=f(\dot{k}^{-1}\dot{h})$ for $f\in \mathcal{A}$ and
$\dot{h},\dot{k}\in D$.
For $m\in \mathcal{A}$, the multiplication operator $M_m$ on
$\oplus_{\dot{h}\in D}L^2(U\times U)$
is defined by
$(M_mg)(\dot{h})=m(\dot{h})\ast g(\dot{h})$
for $\dot{h}\in D, g\in \oplus_{\dot{h}\in D}L^2(U)$
where $\ast$ means convolution of kernels in $L^{\infty}(U\times U)$.
Like in the case where $H$ is a subgroup we have 
$\lambda_{\dot{k}}M_m\lambda_{\dot{k}^{-1}}=M_{T_{\dot{k}}m}$.
\par
Letting $\mathcal{L}=l^1(D,\mathcal{A},T)$ denote the twisted
$L^1$-algebra like in the beginning of section~\ref{sec:subgroup}
(replace $H$ and its elements by $D$ and the respective elements),
we may repeat the proof of Theorem~\ref{thm:CD}
to obtain
\begin{theorem}
The Banach $\ast$-algebra $CD_H$ is
isometrically $\ast$-isomorphic to
$\mathcal{L}=l^1(D,l^{\infty}(D,L^{\infty}(U\times U)),T)$.
As a result $CD_{\infty}$ is bicontinuously $\ast$-isomorphic to $\mathcal{L}$.
\end{theorem}
Continuing as in section~\ref{sec:subgroup} with $D$ in place of $H$
we obtain the analoga of Corollary~\ref{cor:invcl} and Theorem~\ref{thm:invcl}:
\begin{corollary}
$\CC \id + CD_{\infty}(G)$ is inverse-closed in $B(L^2(G))$.
\end{corollary}
\begin{theorem}
$\CC \id + CD(G)$ is inverse-closed in $B(L^2(G))$.
\end{theorem}
\section{Appendix}
Here we shall derive the amenability of a locally compact group 
$G$ satisfying our
assumption (1).
\begin{lemma}
If $K\subset G$ is compact, then it can be covered by finitely
many sets $hU$ with $h\in H$.
\end{lemma}
\begin{proof}
\begin{enumerate}
\item[(i)]{First note that, due to our assumptions, $H$ is not only discrete but
uniformly discrete. This implies that a net $(h_{\mu})$ in 
$H$ cannot converge to a point outside $H$, and if $h_{\mu}\to h\in H$,
then $h_{\mu}=h$
from some index $\mu_0$ onwards.}
\item[(ii)]{Suppose $K\subset G$ is compact and meets infinitely many sets 
$h_nU$, $h_n\in H$, $n\in \NN$, and $h_n\not= h_m$ for $n\not= m$.
For $n\in \NN$ choose $k_n\in K\cap h_nU$, so $k_n=h_nu_n$ with $u_n\in U$.
Choose a convergent subnet $(k_{\mu})$ of $(k_n)_{n\in \NN}$, $k_{\mu}\to k\in K$
say. Choose a subnet $(k_{\mu'})$ of $(k_{\mu})$ such that $k_{\mu'}=h_{\mu'}u_{\mu'}$, with $u_{\mu'}\to \overline{u}\in \overline{U}$. Then $h_{\mu'}=k_{\mu'}u_{\mu'}^{-1}\to k\overline{u}^{-1}$. This contradicts the fact that $h_{\mu'})$, being a subnet of $(h_n)_{n\in \NN}$, cannot converge (according to (i)).
}
\end{enumerate}
\end{proof}
Now we show that $G$ satisfies the F\o lner condition
(see \cite[3.6]{green69}). 
For a compact $K\subset G$  let 
$K$ be covered by finitely many
$h_iU,\; i=1,\ldots ,M$, and let the compact set
$\overline{U^2} \cup \overline{U^{-1}U}$
be covered by $k_jU,\;j=1,\ldots ,L$.
Let $\epsilon>0$ be arbitrary.
Since $H$ satisfies the F\o lner condition, for 
$\cup_{i,j}\{k_jh_i,\, k_jh_i^{-1}\}$
 there exists a nontrivial finite set $E\subset H$, such that
$\abs{k_jh_iE \Delta E}\leq \frac{\varepsilon}{L} \abs{E}$ 
and $\abs{k_jh_i^{-1}E \Delta E}\leq \frac{\varepsilon}{L} \abs{E}$
for $i=1,\ldots ,M$,
$j=1,\ldots, L$. (On the discrete group $H$ the Haar measure coincides
with the counting measure, thus for $E\subset H$ 
the cardinality of $E$ is $\abs{E}$.) 
Then for $x\in K$, $x=h_iu$, say
\begin{eqnarray*}
\abs{h_iuUE\setminus UE}&\leq&\abs{h_iU^2E\setminus UE}\;=\;%
\abs{U^2h_iE\setminus UE}\\
&\leq&\abs{\bigcup_{j=1}^Lk_jUh_iE \setminus UE}\;=\;%
\abs{U}\abs{\bigcup_{j=1}^Lk_jh_iE\setminus E}\\
&\leq& \abs{U}\sum_{j=1}^L\abs{k_jh_iE\setminus E}\;\leq\;%
\abs{U}\varepsilon\abs{E}\;=\;\varepsilon\abs{UE}.
\end{eqnarray*}
Similarly
\begin{eqnarray*}
\abs{UE\setminus h_iuUE}&\leq&\abs{u^{-1}h_{i}^{-1}UE \setminus UE}\;\leq\;%
\abs{U^{-1}Uh_{i}^{-1}E \setminus UE}\\
&\leq&\abs{\bigcup_{j=1}^Lk_jUh_{i}^{-1}E \setminus UE}\;\leq\;%
\abs{U}\sum_{j=1}^L\abs{k_jh_{i}^{-1}E \setminus E}\\
&\leq&\abs{U}\varepsilon\abs{E}\;=\;\varepsilon\abs{UE}
\end{eqnarray*}
So we obtain, for $x=h_iu\in K$
\begin{eqnarray*}
\abs{xUE\Delta UE}&=&\abs{\{h_iuUE\setminus UE\}\cup \{UE\setminus h_iuUE\}}\\%
&\leq&\abs{h_iuUE\setminus UE}+\abs{UE\setminus h_iuUE}\;\leq\;%
2\varepsilon\abs{UE}.
\end{eqnarray*}
\section{Acknowledgements}
We thank Karlheinz~Gr\"ochenig for his careful 
reading
of a preliminary version of this note, and Karl~H.~Hofmann for a 
discussion on the structure of the groups we consider.
We thank Michael~Cowling for a fruitful discussion on
the amenability of these groups.
\bibliographystyle{amsplain}

\end{document}